\documentclass[12pt,reqno]{amsart}

\usepackage{multirow}
\usepackage{hhline}

\usepackage{mathtools}
\usepackage{times}
\usepackage[T1]{fontenc}
\usepackage{mathrsfs}
\usepackage{latexsym}
\usepackage[dvips]{graphics}
\usepackage[titletoc, title]{appendix}
\usepackage{amsmath,amsfonts,amsthm,amssymb,amscd}
\usepackage[dvipsnames]{xcolor}
\usepackage{hyperref}
\usepackage{amsmath}
\usepackage[utf8]{inputenc}

\usepackage{color}
\usepackage{breakurl}

\usepackage{comment}
\newcommand{\bburl}[1]{\textcolor{blue}{\url{#1}}}
\newcommand{\seqnum}[1]{\href{https://oeis.org/#1}{\rm \underline{#1}}}

\theoremstyle{plain}
\numberwithin{equation}{section}
\newtheorem{thm}{Theorem}[section]

\newtheorem{proposition}[thm]{Proposition}

\usepackage[utf8]{inputenc}

\DeclareFixedFont{\ttb}{T1}{txtt}{bx}{n}{12} % for bold
\DeclareFixedFont{\ttm}{T1}{txtt}{m}{n}{12}  % for normal

\usepackage{color}
\definecolor{deepblue}{rgb}{0,0,0.5}
\definecolor{deepred}{rgb}{0.6,0,0}
\definecolor{deepgreen}{rgb}{0,0.5,0}

\usepackage{listings}

\newcommand\pythonstyle{\lstset{
language=Python,
basicstyle=\ttm,
morekeywords={self},              % Add keywords here
keywordstyle=\ttb\color{deepblue},
emph={MyClass,__init__},          % Custom highlighting
emphstyle=\ttb\color{deepred},    % Custom highlighting style
stringstyle=\color{deepgreen},
frame=tb,                         % Any extra options here
showstringspaces=false
}}

\lstnewenvironment{python}[1][]
{
\pythonstyle
\lstset{#1}
}
{}

\newcommand\pythoninline[1]{{\pythonstyle\lstinline!#1!}}

\numberwithin{equation}{section}

\DeclareFontFamily{U}{mathx}{}
\DeclareFontShape{U}{mathx}{m}{n}{<-> mathx10}{}
\DeclareSymbolFont{mathx}{U}{mathx}{m}{n}
\DeclareMathAccent{\widehat}{0}{mathx}{"70}
\DeclareMathAccent{\widecheck}{0}{mathx}{"71}

\begin{document}

\title{Composite Numbers in an Arithmetic Progression}

\author{H\`ung Vi\d{\^e}t Chu}
\address{Department of Mathematics, Texas A\&M University, College
  Station, TX 77843, USA}
  \email{hungchu1@tamu.edu}

\author{Steven J. Miller}
\email{\textcolor{blue}{\href{mailto:sjm1@williams.edu}{sjm1@williams.edu},
\href{mailto:Steven.Miller.MC.96@aya.yale.edu}{Steven.Miller.MC.96@aya.yale.edu}}}
\address{Department of Mathematics and Statistics, Williams College, Williamstown, MA 01267, USA}

\author{Joshua M. Siktar}
\address{Department of Mathematics, Texas A\&M University, College
  Station, TX 77843, USA}
  \email{jmsiktar@tamu.edu}

\thanks{}

\subjclass[2020]{97D50 (primary), 97D60, 11-02}

\keywords{problem writing; student assessment; Dirichlet's Theorem, density of primes}

\begin{abstract}
One challenge (or opportunity!) that many instructors face is how varied the backgrounds, abilities, and interests of students are. In order to simultaneously instill confidence in those with weaker preparations and still challenge those able to go faster, an instructor must be prepared to give problems of different difficulty levels. Using Dirichlet's Theorem as a case study, we create and discuss a family of problems in number theory that highlight the relative strengths and weaknesses of different ways to approach a question and show how to invite students to extend the problems and explore research-level mathematics.
\end{abstract}

\maketitle
\tableofcontents

%%%%%%%%%%%%%%%%%%%%%%%%%%%%%%%%%%%%%%%%%%%%%%%%%%%%%%%%%%%%%%%%%%%%%%%%%%%%%%%%%%%%%%%%%%%%%%%%%%%%%%%%%%%%%%%%%%%%%%%%%%%%%%%%%%%%%%%%%%%%%%%%%%%%%%%%%%%%%%%%%%%%%%%%%%%%%%%%%%%%%%%%%%%%%%%%%%%%%%%%%%%%%%%%%%%%%%%%%%%%%%%%%%%%%%%%%%%%%%%%%%%%%%%%%%%%%%%%%%%%%%%%%%%%%%%%%%%%%%%%%%%%%%%%%%%%%%%%%%%%%%%%%%%%%%%%%%%%%%%%%%%%%%%%%%%%%%%%%%%%%%%%%%%%%%%%%%%%%%%%%%%
\section{Introduction}\label{sec: probintro}

\tableofcontents
\textit{Even fairly good students, when they have obtained the solution of the problem and written down neatly the argument, shut their books and look for something else. Doing so, they miss an important and instructive phase of the work. … A good teacher should understand and impress on his students the view that no problem whatever is completely exhausted. One of the first and foremost duties of the teacher is not to give his students the impression that mathematical problems have little connection with each other, and no connection at all with anything else. We have a natural opportunity to investigate the connections of a problem when looking back at its solution.}

\vspace{0.3 cm}

This is a quote from George P\'{o}lya's famous book \textit{How to Solve it} \cite{polya2014solving} that encapsulates the problem-solving spirit we hope teachers will instill in their students after reading this article.

As we travel from school to school, it's important to realize that what works well in one institution may not in another. The genesis of this article is from the experiences of one of the authors\footnote{This is one reason we have multiple people writing this, to protect identities! The other is that we've  all known each other for years, working and talking with each other all the time.} who at one point in his career moved from schools where it was the norm to have a 50\% be an excellent score in an upper level exam to a place where such an average led to students complaining and doubting their ability to do mathematics! Being told that there was a curve did no good, nor comparisons such as getting a hit in one out of every three at-bats would have you ranked as one of the greatest ballplayers of all time. The students worried that the numerical value indicated they could not do mathematics.

There are many solutions to such a situation. The chosen one was to add questions to assignments / exams which were hopefully straightforward and would be easily solved by the entire class and decrease the point count of the other problems, thus effectively baking in the curve. For this to work, though, it's necessary to find problems that everyone will get! We focus on one such problem for several reasons. First, for colleagues just starting out, it highlights the creative process in building an exam. Second, over the years this problem has been discussed with several people, and recently these conversations led to intriguing explorations on alternate ways to solve the problem, and future non-trivial questions one can ask, highlighting how one can find research problems everywhere! We thus decided to write up our dialogue, demonstrating that there are often multiple ways of solving a problem, each with its strengths and weaknesses\footnote{See Terry Tao's blogpost \cite{tao2008limitations} that discusses this in more detail.}.

\vspace{0.3 cm}

So what's our problem of focus? There are many beautiful results in number theory; a particularly striking one is Dirichlet's Theorem.\footnote{See \url{https://www.youtube.com/watch?v=zG185Ef1gPM} for a video lecture on the motivation behind Dirichlet's Theorem.} Recall every positive integer falls into one of three mutually distinct categories: 1 is the unit, an integer $p > 1$ is prime if it has precisely two factors (itself and 1), and all other numbers are composites. By the Fundamental Theorem of Arithmetic (see the classic \cite{HW}, which is full of gems to read), every integer can be factored uniquely as a product of primes in increasing order, and this is why we do not want one to be considered a prime. If that were true then we would no longer have unique factorization due to the ability to add any number of factors of 1.

\begin{thm}[Dirichlet] If $a$ and $b$ are relatively prime\footnote{Sometimes one assumes that our numbers are positive; we leave it to the reader to show that such an assumption is harmless.} then there are infinitely many primes congruent to $b$ mod $a$. In other words, the sequence $an + b$ is prime for infinitely many $n$.
\end{thm}

What's fascinating is that the proof used is the same one that's needed for a \emph{much} easier question: prove there is \emph{at least one prime} congruent to $b$ modulo $a$ when $a$ and $b$ are relatively prime! That's right -- the way we know how to prove there's at least one such prime is to prove there are infinitely many!\footnote{It's worse than that. If $\pi(x)$ is the number of primes up to $x$, and $\pi_{a,b}(x)$ the number of primes $p \le x$ congruent to $b$ modulo $a$, then Dirichlet's Theorem is proved by showing to first order $\pi_{a,b}(x) = \pi(x)/\varphi(b)$, with $\varphi$ being the Euler totient function counting the number of integers at most $n$ relatively prime to $n$. In other words, to show there is \emph{one} prime congruent to $b$ modulo $a$ we show there are infinitely many, and the way we do that is we show all possible residue classes $b$ mod $a$ (i.e., those $b$ relatively prime to $a$) to first order have the same number of primes up to $x$ as $x \to \infty$!}

A moment's reflection shows the conditions are necessary; if $a$ and $b$ aren't relatively prime, then there's at most one prime congruent to $b$ modulo $a$, namely $\gcd(a,b)$, if $\gcd(a,b)$ is prime. The difficulty is proving this sufficient condition is necessary.

One of us thus proposed the following problem, attributed to the aptly named fictional mathematician Telhcirid. 
\begin{proposition}[Telhcirid]\label{Prop: Telhcirid}
If $a$ and $b$ are integers (with $a$ nonzero), then there are infinitely many composites congruent to $b$ modulo $a$.
\end{proposition}

In the remaining sections, we discuss various proofs, the benefits and disadvantages of the different approaches, and describe how to build research questions from these conversations.

%%%%%%%%%%%%%%%%%%%%%%%%%%%%%%%%%%%%%%%%%%%%%%%%%%%%%%%%%%%%%%%%%%%%%%%%%%%%%%%%%%%%%%%%%%%%%%%%%%%%%%%%%%%%%%%%%%%%%%%%%%%%%%%%%%%%%%%%%%%%%%%%%%%%%%%%%%%%%%%%%%%%%%%%%%%%%%%%%%%%%%%%%%%%%%%%%%%%%%%%%%%%%%%%%%%%%%%%%%%%%%%%%%%%%%%%%%%%%%%%%%%%%%%%%%%%%%%%%%%%%%%%%%%%%%%%%%%%%%%%%%%%%%%%%%%%%%%%%%%%%%%%%%%%%%%%%%%%%%%%%%%%%%%%%%%%%%%%%%%%%%%%%%%%%%%%%%%%%%%%%%%%%%%%%%%

\section{Solution}

Unlike in Dirichlet's Theorem, we don't need $a$ and $b$ to be relatively prime in Proposition \ref{Prop: Telhcirid}. There are many ways we could prove this. One option is to use the following result, for if we can show that in the limit, zero percent of integers are prime, then any arithmetic progression must in the limit have almost all of its numbers composite.\footnote{Note there's a difference between 100\% and all. In the limit 100\% of integers are not squares: the percent of integers at most $x$ that are not squares is approximately $(x - x^{1/2})/x = 1 - x^{-1/2}$; however, there are infinitely many squares.}

\begin{proposition}[Zero density of primes]\label{Prop: ZeroDense}
    The set of prime numbers has zero density. That is, if $\pi(x)$ is the number of prime numbers at most $x$, then
    \begin{equation}\label{Eq: ZeroDensity}
        \lim_{x \rightarrow \infty}\frac{\pi(x)}{x} \ = \ 0.
    \end{equation}
\end{proposition}

We give the proof in Section \ref{sec:proofpropzerodense}; if you know what quantities to study it's not hard, but there is a bit of an art to figuring out what should be studied. Rather than going down that technical detour now, we see why it's a worthwhile route to explore.

Our goal, Proposition \ref{Prop: Telhcirid} (there are infinitely many composites congruent to $b$ modulo $a$) follows readily from Proposition \ref{Prop: ZeroDense}.

\begin{proof}[(First) proof of Proposition \ref{Prop: Telhcirid}]
Assume Proposition \ref{Prop: Telhcirid} is false for some integers $a$ and $b$. Then there exists an integer $K$ such that, for all $k > K$, $ak + b$ is prime, which implies the primes have positive density (of at least $1/|a|$; of course $a$ cannot be zero!).\end{proof}

\ \\

This approach is far stronger than we need; it yields in the limit that 100\% of the inputs yield composite numbers. Further, it doesn't give much insight into why Proposition \ref{Prop: Telhcirid} is true. So we continue our search for a simpler proof of the weaker result, namely that there are infinitely many composites in the sequence (and not that almost all inputs give composites). The next, more direct proof, better illuminates the present difficulties.

\begin{proof}[(Second) proof of Proposition \ref{Prop: Telhcirid}]
    Suppose first that $|b| > 1$ (so it's not $\pm 1$). Then the subsequence $\{a(bm) + b\}^{\infty}_{m = 1}$ is entirely comprised of composite numbers.\footnote{Most students have no trouble finding this explicit example. Note that it gives infinitely many composites, but it doesn't give that almost all inputs are composite!}
    
    Now we consider the case where $b = \pm 1$. It's natural to try to find a special sequence again, but this time it's not as clear what to try, and many students over the years missed this part.  
    
    We want to construct a sequence of composites that are each $1$ away from some multiple of $a$. We use the method of \textit{divine inspiration} \cite{MiP}: that is, we give a choice of $n$ where $an+b$ can be written as a product of two integers greater than 1, and once we can verify that works we are done! (We'll explain later how we came to be so inspired.)

     Choosing $n = a(am + b)+m$ gives
    \begin{equation*}
        an + b \ = \ a(a(am+b)+m) + b \ = \ (a^2+1)(am + b),
    \end{equation*}
    Since $a\neq 0$, we know that $a^2 + 1\ge 2$. For sufficiently large $m$, $|am + b| > 1$, making $an + b$ a composite. 
\end{proof}

Let's briefly discuss how we knew to try $n = a(am+b)+m$ in the above proof. We write $n$ as $\xi + \gamma$, where $\xi$ and $\gamma$ are free. We want to be able to factor
$$a(\xi + \gamma) + b\ =\ a\xi + (a\gamma+b),$$
which is possible if $\xi$ is a multiple of $(a\gamma+b)$; while there are many multiples we can choose, a good one is $a(a\gamma+b)$ as the $a$ reinforces the $a$ we have, ensuring positive quantities. In that case,
$$an + b\ =\ a^2(a\gamma+b) + (a\gamma+b)\ =\ (a^2+1)(a\gamma+b).$$
Now let $\gamma = m$. By no means do we claim that we give the \textit{only} valid construction for this proof technique; to a large extent, we simply follow our noses. This goes to show the importance of one's intuition when solving problems.

The third time, perhaps, is the charm. Here is yet another proof. The idea is to use extensions of $$x^2 - y^2 \ =\  (x - y) (x + y) \ \ \ {\rm and} \ \ \ x^3 + y^3 \ = \ (x + y)(x^2 - xy + y^2).$$

\begin{proof}[Third proof of Proposition \ref{Prop: Telhcirid}]
If $|b| > 1$, then $an + b$ is composite whenever $n$ is a multiple of $b$. Suppose that $b = \pm 1$. 
Without loss of generality, assume $a \ge 1$ (otherwise, we consider $-an-b$).
Our arithmetic progression is then $an \pm 1$. We handle the two cases separately.

For $n = 3^{2k+1}a^{2k}$ ($k\in \mathbb{N}$), we have
$$an - 1\ =\ (3a)^{2k+1} - 1\ =\ (3a-1)((3a)^{2k} + (3a)^{2k-1} + \cdots + 1).$$
Since $3a-1$ and $(3a)^{2k} + (3a)^{2k-1} + \cdots + 1$ are greater than $1$, we know that $an-1$ is composite. 
The same reasoning applies to
$$an + 1\ =\ (3a)^{2k+1} + 1\ =\ (3a+1)((3a)^{2k} - (3a)^{2k-1} + \cdots + 1).$$
\end{proof}

%%%%%%%%%%%%%%%%%%%%%%%%%%%%%%%%%%%%%%%%%%%%%%%%%%%%%%
%%%%%%%%%%%%%%%%%%%%%%%%%%%%%%%%%%%%%%%%%%%%%%%%%%%%%%
%%%%%%%%%%%%%%%%%%%%%%%%%%%%%%%%%%%%%%%%%%%%%%%%%%%%%%
\section{Asking more refined questions}\label{Sec: refined} 

We shift gears and discuss creating new questions which can hopefully be answered by students (either as another exam question or as a research project), using the proof ideas of Proposition \ref{Prop: Telhcirid} as inspiration. A general maxim for generating such problems is the following: If one successfully produces a qualitative result (e.g., ``there exists,'' ``for some $n$,'' ``happens infinitely often,'' etc.), one then tries to produce quantitative results carrying more information (e.g., ``when $n$ is equal to,'' ``with the density of,'' ``is bounded by,'' etc.). In our attempt to solve a more refined problem, we may pleasantly encounter alternative solutions to the original one.  We've seen both qualitative and quantitative proofs of Proposition \ref{Prop: Telhcirid}, and now we illustrate this practice further. 

We may ask, for a given positive integer $N$, are there always $N$ \textit{consecutive} composites in our arithmetic progressions $an+b$ $(a > 0)$? The answer is yes, thanks to the zero density of primes (Proposition \ref{Prop: ZeroDense}) because otherwise the density of primes would be at least $1/aN > 0$, a contradiction. Of course, this argument is not entirely elementary; is there an easier way to see this?

Yes! We can avoid using Proposition \ref{Prop: ZeroDense} by observing that for each positive integer $m > 2$, all integers in $S_m = \{m!+2, m!+3, \ldots, m!+m\}$ are composite. Since $|S_m| = m$, if we choose $m$ sufficiently large\footnote{How large must $m$ be in terms of $a$ and $N$?} relative to $a$, we can ensure we have $N$ consecutive terms in our arithmetic progression are in $S_m$. \\ \

For another question: What proportion of natural numbers $n$ generate composite numbers in our progression $an + b$? In other words, as $x \to \infty$ what fraction of $n \le x$ lead to $an+b$ is composite? The density of such $n$ is $1$, and again we can see this by appealing to our result that the density of the primes is zero, but let's take this opportunity to discuss how to elementarily prove Proposition \ref{Prop: ZeroDense}. 

There are two possibilities: there are only finitely many primes, or there are infinitely many primes. If there are only finitely many primes in the sequence, then trivially the proportion of $n$ making $an+b$ converges to 100\%. 

Thus we're left with the case when there are infinitely many primes; we want to use as few properties of the primes as we can get away with. The idea is to use Inclusion - Exclusion. We'll leave it as an exercise to prove, but be warned: it is very easy to make a small assumption and have an invalid argument; one has to be very careful in doing such arguments.\footnote{Or see Tao's post at \bburl{https://mathoverflow.net/questions/381456/what-is-the-simplest-proof-that-the-density-of-primes-goes-to-zero}.}

Last but not least, we can, of course, ask for the longest run of primes in our arithmetic progression $an+b$ as $n\rightarrow\infty$. Since 
$$a(a(am+b) + m) + b\ =\ (a^2+1)(am+b),$$
and $a^2+1, am+b > 1$ for sufficiently large $m\in \mathbb{N}$, we know that $an + b$ is composite when $n = a(am+b) + m$. Therefore, as $n\rightarrow\infty$, a run of primes cannot be longer than $a^2$. This bound is sharp when $a = 1$, but not when $a = 2$ and $b = 1$, for example. To see why, note that $2k+1, 2k+3$, and $2k+5$ form a complete residue system modulo $3$, so exactly one of them is divisible by $3$. If they're all primes, then one of them is $3$. This happens only when $k = 1$, in which case neither $2k-1$ nor $2k+7$ is a prime. Hence, as $n\rightarrow\infty$, the longest run of primes in $2n+1$ is $2$. Meanwhile, our bound is $a^2 = 4$. We leave it to the reader to investigate the problem of the longest run of primes in an arithmetic progression as $a$ and $b$ vary! 

%%%%%%%%%%%%%%%%%%%%%%%%%%%%%%%%%%%%%%%%%%%%%%%%%%%%%%%%%%%%%%%%%%%%%%%%%%%%%%%%%%%%%%%%%%%%%%%%%%%%%%%%%%%%%%%%%%%%%%%%%%%%%%%%%%%%%%%%%%%%%%%%%%%%%%%%%%%%%%%%%%%%%%%%%%%%%%%%%%%%%%%%%%%%%%%%%%%%%%%%%%%%%%%%%%%%%%%%%%%%%%%%%%%%%%%%%%%%%%%%%%%%%%%%%%%%%%%%%%%%%%%%%%%%%%%%%%%%%%%%%%%%%%%%%%%%%%%%%%%%%%%%%%%%%%%%%%%%%%%%%%%%%%%%%%%%%%%%%%%%%%%%%%%%%%%%%%%%%%%%%%%%%%%%%%%%%%%%%%%%%%%%%%%%%%%%%%%%%%%%%%%%%%%%%%%%%%%%%%%%%%%%%%%%%%%%%%%%%%%%%%%%%%%%%%%%%%%%%%%%%%%%%%%%%%%%%%%%%%%%%%%%%%%%%%%%%%%%%%%%%%%%%%%%%%%%%%%%%%%%%%%%%%%%%%%%%%%%%%%%%%%%%%%%%%%%%%%%%%%%%%%%%%%%%%%%%%%%%%%%%%%%%%%%%%%%%%%%%%%%%%%%%%%%%%%%%%%%%%%%
\section{Relating our problem to advanced results} 

Now that we've given the reader ample groundwork to play with extensions and variants of Dirichlet's (and Telhcirid's) Theorem, we demonstrate the importance of this mathematical play by connecting our discussion to more sophisticated results in number theory. Indeed, as educators, we should show students the endless possibilities of extending a problem and the benefits of doing so, including the following.

\begin{enumerate}
\item[a)] After several extensions of a problem, students will likely encounter research-level mathematics. The urge to solve an extended version of a problem they previously solved serves as a strong motivation to study advanced results and make new 
contributions. 

\item[b)] We may also give students a relatively challenging problem, possibly allow them to assume an advanced result, and tell them to use what they just learned in the course to solve the problem in a few lines. This fosters their understanding of the class material and helps portray its significance in a broader mathematical context, and hopefully motivates them to continue onward and understand the deeper results.
\end{enumerate}

Let's illustrate item a) first. To extend a problem, a common method is to review the ideas used in its solution and check if we've exploited their power to the fullest. For example, combining Dirichlet's Theorem on arithmetic progressions and explicit constructions of composites, we can ask about $k$-composites; these are numbers which are the product of exactly $k$ distinct primes.

The problem statement above illustrates a common issue: it's not always clear what's the right generalization or definition. Do we want $k$ distinct primes to divide our number, or do we want exactly $k$ primes? For example, $12 = 2^2 3$; is this a 2-composite number or should this be a 3-composite number as it's divisible by 2 twice and 3 once? Either definition is fine, we just have to be clear what we mean; frequently such introspection leads to related problems to study! We'll figure out what we mean to study as we start exploring; this is not uncommon - until we dive in to the problem it's not clear what the right definition should be.

Thus, let's try to solve the following question (and in the course of doing so figure out what we mean by $k$-composite!): \emph{For $a$ and $b$ relatively prime, are there infinitely many $k$-composites in an arithmetic progression $an+b$?} 

We prove that the answer is yes. The case $k = 1$ follows immediately from Dirichlet's Theorem, which we'll freely use. For a proof, see \cite{Da, Se}. Let's consider $k = 2$. From the $k=1$ case applied to the relatively prime pair $a^2$ and 1 we can choose a positive integer $s$ such that $sa^2+1$ is prime. Appealing once again to the Method of Divine Inspiration (you're encouraged to discover the path we took to such enlightenment) we take $n = sa(am+b) + m$ with $m$ free in our arithmetic progression, and note $an+b$ becomes 
\begin{equation}\label{e1}a(\underbrace{sa(am+b)+m}_{n_s(m)}) + b\ =\ (sa^2+1)(am+b).\end{equation}
Since $am+b$ is prime for infinitely many $m$ by Dirichlet's Theorem, $an_s(m)+b$ is $2$-composite infinitely often. At this point we now have an argument for what we mean by $k$-composite. It's possible that $sa^2+1$ and $am+b$ are the same prime, so perhaps we want to have exactly two prime factors, and they can be the same. Of course, since we have infinitely many $m$ such that $am+b$ is prime it's easy to force the two factors to be distinct primes; if they happen to be the same just take a larger $m$, which we know we can do by Dirichlet's Theorem.

A moment of reflection reveals that  \eqref{e1} represents certain numbers in an arithmetic progression as a product of a prime and numbers in the same arithmetic progression. This allows us to use induction to prove the infinitude of $k$-composites. Indeed, for each time $am+b$ is $2$-composite, we pick a sufficiently large $s$ such that $sa^2+1$ is a prime larger than $am+b$. Then $(sa^2+1)(am+b)$ is $3$-composite. 

Going further, we can ask about the density of $k$-composites in our arithmetic progression, which turns out to be $0$, though not surprisingly we're leaving elementary results and now appealing to beautiful theorems in the literature. To see this, we recall the beautiful Erd\H{o}s-Kac theorem \cite{EK}, which states that for any fixed $a < b$,
$$\lim_{x\rightarrow\infty} \frac{1}{x}\cdot \left|\left\{n\le x\,:\, a\le \frac{\omega(n)-\log\log n}{\sqrt{\log \log n}} \le b\right\}\right| \ =\ \frac{1}{\sqrt{2\pi}}\int_{a}^be^{-t^2/2}dt.$$
Consequently, almost all numbers $n$ have $\log\log n$ distinct prime factors. \\ \

Another direction is to investigate what happens when we replace $an+b$ by a non-constant polynomial $f(n)$ with integer coefficients. We can still easily and elementarily show that $f(n)$ is composite for infinitely many $n$. Write $f(x) = a_m x^m + a_{m-1}x^{m-1} +\cdots a_1x + a_0$ for some $m\ge 1$ and $a_m\neq 0$. Without loss of generality, assume that $a_m \ge 1$. Since $\lim_{x\rightarrow \infty} f(x) = \infty$, there exists $k\in \mathbb{N}$ such that $f(k) > 1$. Then for $j \in \mathbb{N}$, 
$$f(k + jf(k))\ =\ a_m (k+jf(k))^m + a_{m-1}(k+jf(k))^{m-1} +\cdots +a_1(k + jf(k)) + a_0$$ is a multiple of $f(k)$ and is greater than $f(k)$ for sufficiently large $j$. Hence, $f(k+jf(k))$ is composite and thus $f(n)$ is composite infinitely often.

While the above reasoning shows that we cannot have a polynomial which is always prime, it would be a disservice if we didn't mention Euler's function $n^2+n+41$, which produces primes for $0\le n\le 39$! Fendel \cite{Fe} connected polynomials of the form $n^2+n+C$ that produce ``many'' primes to principal ideal domains. On another related note, a positive integer $C$ is said to be an \textit{Euler's lucky number} if $n^2-n+C$ produces primes for all $1\le n\le C$. It follows from the work of Heegner \cite{He} and Stark \cite{St} that there are only six Euler's lucky numbers, namely 2, 3, 5, 11, 17, and 41 \cite{WM}. They form the sequence \seqnum{A014556} in the Online Encyclopedia of Integer Sequences (OEIS) \cite{Sl}. \\ \

Regarding item b), we can show students an application of the difference-of-squares formula through the following problem: Prove directly the infinitude of $3$-composites in the arithmetic progression $4n+3$, assuming the Twin Prime Conjecture (there are infinitely many primes $p$ such that $p+2$ is also prime). A quick solution is to observe that for $n = 5k^2-2$, 
$$4n+3 \ =\ 4(5k^2-2) + 3\ =\ 20k^2-5\ =\ 5(4k^2-1)\ =\ 5(2k-1)(2k+1).$$
Under the twin prime conjecture, the pair $2k-1$ and $2k+1$ are simultaneously primes infinitely often, and we are done. Can this argument be extended to prove there are infinitely many $4$-composites by assuming a conjecture about numbers, and if so what's the weakest conjecture you need to assume?

We end with one final problem, which gives a nice application of the Intermediate Value Theorem (IVT). Fermat's Last Theorem states that given an integer $r > 2$, there are no positive integers $x, y, z$ to the equation $x^r + y^r = z^r$. Does the conclusion still hold if $r$ is a real number? To solve this problem (see \cite{Mo}), consider the function $f(x) = 4^x+5^x-6^x$. Clearly, $f(2) > 0$ and $f(3) < 0$. The IVT guarantees the existence of $s\in (2, 3)$ such that $4^s+5^s-6^s = 0$; what kind of number is $s$? Thanks to Lenstra \cite{Le}, we know $s$ must be irrational, but it's unknown if it's transcendental.\footnote{If $\alpha$ is a root of a finite polynomial with integer coefficients then it is algebraic, otherwise it is transcendental (see \cite{MT-B}). Numbers such as $\sqrt{2}$ and more interestingly $\sqrt{\sqrt{2} + \sqrt{3}}$ are algebraic, while $e$ and $\pi$ are transcendental.} Perhaps the answer depends on what triple we use - could different triples lead to different types of $s$?

%%%%%%%%%%%%%%%%%%%%%%%%%%%%%%%%%%%%%%%%%%%%%%%%%%%%%%%%%%%%%%%%%%%%%%%%%%%%%%%%%%%%%%%%%%%%%%%%%%%%%%%%%%%%%%%%%%%%%%%%%%%%%%%%%%%%%%%%%%%%%%%%%%%%%%%%%%%%%%%%%%%%%%%%%%%%%%%%%%%%%%%%%%%%%%%%%%%%%%%%%%%%%%%%%%%%%%%%%%%%%%%%%%%%%%%%%%%%%%%%%%%%%%%%%%%%%%%%%%%%%%%%%%%%%%%%%%%%%%%%%%%%%%%%%%%%%%%%%%%%%%%%%%%%%%%%%%%%%%%%%%%%%%%%%%%%%%%%%%%%%%%%%%%%%%%%%%%%%%%%%%%%%%%%%%%%%%%%%%%%%%%%%%%%%%%%%%%%%%%%%%%%%%%%%%%%%%%%%%%%%%%%%%%%%%%%%%%%%%%%%%%%%%%%%%%%%%%%%%%%%%%%%%%%%%%%%%%%%%%%%%%%%%%%%%%%%%%%%%%%%%%%%%%%%%%%%%%%%%%%%%%%%%%%%%%%%%%%%%%%%%%%%%%%%%%%%%%%%%%%%%%%%%%%%%%%%%%%%%%%%%%%%%%%%%%%%%%%%%%%%%%%%%%%%%%%%%%%%%%%

\section{Conclusion}

Presenting students with problems of increasing difficulty levels gives everyone a taste of success, builds up their confidence and persistence, and effectively improves their problem-solving skills. During the process, students may encounter new challenges that are great opportunities for us as educators to teach them how to respond to setbacks and impasses. Furthermore, these challenges sometimes result in a research problem that is both accessible to the undergraduates and interesting to the research community. The authors hope that this article gives a sense of the creative process in finding problems to engage students both inside the classroom and in research settings. We would like to end this article with a quote by G. Cantor: 

\vspace{0.3 cm}

\textit{In mathematics, the art of proposing a question must be held of higher value than solving it.}

%%%%%%%%%%%%%%%%%%%%%%%%%%%%%%%%%%%%%%%%%%%%%%%%%%%%%%%%%%%%%%%%%%%%%%%%%%%%%%%%%%%%%%%%%%%%%%%%%%%%%%%%%%%%%%%%%%%%%%%%%%%%%%%%%%%%%%%%%%%%%%%%%%%%%%%%%%%%%%%%%%%%%%%%%%%%%%%%%%%%%%%%%%%%%%%%%%%%%%%%%%%%%%%%%%%%%%%%%%%%%%%%%%%%%%%%%%%%%%%%%%%%%%%%%%%%%%%%%%%%%%%%%%%%%%%%%%%%%%%%%%%%%%%%%%%%%%%%%%%%%%%%%%%%%%%%%%%%%%%%%%%%%%%%%%%%%%%%%%%%%%%%%%%%%%%%%%%%%%%%%%%%%%%%%%%%%%%%%%%%%%%%%%%%%%%%%%%%%%%%%%%%%%%%%%%%%%%%%%%%%%%%%%%%%%%%%%%%%%%%%%%%%%%%%%%%%%%%%%%%%%%%%%%%%%%%%%%%%%%%%%%%%%%%%%%%%%%%%%%%%%%%%%%%%%%%%%%%%%%%%%%%%%%%%%%%%%%%%%%%%%%%%%%%%%%%%%%%%%%%%%%%%%%%%%%%%%%%%%%%%%%%%%%%%%%%%%%%%%%%%%%%%%%%%%%%%%%%%%%%

\appendix

\section{Proof of Proposition \ref{Prop: ZeroDense}}\label{sec:proofpropzerodense}

\begin{proof}[Proof of Proposition \ref{Prop: ZeroDense}]
We reproduce  the proof given in \cite{mercerdensity}, though there are many others. 
By the Binomial Theorem, $$(x + y)^m \ = \ \sum_{k=0}^m {{m}\choose{k}} x^k y^{m-k}, \ \ \ {\rm with}\ \ \ \ {{m}\choose{k}} \ = \ \frac{m!}{k!(m-k)!};$$
remember that $\frac{m!}{k!(m-k)!}$ represents the number of ways to choose $k$ from $m$ distinct objects, where order does not matter, and therefore is an integer. 

Let's take $m = 2n$ and $x=y=1$, thus $$ \sum_{k=0}^m {{2n}\choose{k}} \ = \ 2^{2m}  \ = \ 4^n.$$ As all terms on the left are positive, looking at just the central term we see ${{2n}\choose{n}} < 4^n$. In addition, since ${{2n}\choose{n}} = \frac{(2n)!}{n!n!}$ is an integer, it must be a multiple of all the primes larger than $n$ as none of those primes divide the integers in the denominator. Thus each prime in the range $n + 1, n + 2, \dots, 2n$ divides this central binomial coefficient.\footnote{As an interesting aside, there's at least one prime number in this range; this result is known as Bertrand's postulate; see for example \cite{MT-B}.} Therefore, ${{2n}\choose{n}} > n^{\pi(2n) - \pi(n)}$, and since the central binomial coefficient is less than $4^n$, we have 
    \begin{equation}\label{Eq: coeffBound}
        n^{\pi(2n) - \pi(n)} \ < \ 4^n.
    \end{equation}
    By applying the natural logarithm to both sides of the inequality, we obtain
    \begin{equation*}
        \pi(2n) - \pi(n) \ <\ \frac{n\ln 4}{\ln n}.
    \end{equation*}

   We now use a standard trick: dyadic decomposition \cite{MiP}. Basically, we apply this observation many times with different choices of $n$, stitching them together to get the desired result. Specifically, if  we choose $n = 2^{k-1}$ for some $k \in \mathbb{N}$, we have that
    \begin{equation}\label{Eq: piSubtract}
\pi(2^{k}) - \pi(2^{k-1}) \ < \ \frac{2^{k - 1}\ln 4}{(k - 1)\ln 2} \ = \ \frac{2^k}{k - 1}.
    \end{equation}
    Summing over all $k$ from $2$ to $2m$ causes telescoping on the left-hand side of \eqref{Eq: piSubtract} and then
    \begin{align*}
        \pi(2^{2m}) - \pi(2) &\ < \ \left(\frac{2^2}{1} + \cdots + \frac{2^{m}}{m-1}\right) + \left(\frac{2^{m+1}}{m} + \cdots + \frac{2^{2m}}{2m-1}\right)\\
        &\ < \ 2^{m+1} + \frac{2^{m+1} + \cdots + 2^{2m}}{m}\ <\ 2^{m+1} + \frac{2^{2m+1}}{m}.
    \end{align*}
    Hence, 
    $$\pi(4^m) \ <\ 1 + 2^{m+1} + \frac{2^{2m+1}}{m}.$$
    A positive integer $x$ lives between two consecutive powers of $4$; given any $x$ there is an $m$ such that 
    $4^{m-1}<x\le 4^m$; we can re-write this as 
    \begin{equation*}
        m - 1 \ < \ \log_4 x \ \le \ m.
    \end{equation*}
    Choosing $m$ as appropriate for a given $x$, we have the inequality
    \begin{equation*}
        \pi(x) \ \leq \ \pi(4^m) \ < \ 1 + 2^{(1 + \log_4 x) + 1} + \frac{2^{2(1 + \log_4 x) + 1}}{\log_4 x}.
    \end{equation*}
       After simplifying and rearranging, we get 
    \begin{equation*}
        \frac{\pi(x)}{x} \ < \ \frac{1}{x} + \frac{4}{\sqrt{x}} + \frac{8}{\log_4 x},
    \end{equation*}
    and the right-hand side goes to $0$ as $x \rightarrow \infty$, completing the proof.
\end{proof}

This is only one of many proofs;  we encourage the reader to look for others and discuss with friends and classmates. Notice also that some steps like \eqref{Eq: coeffBound} are crude and could be made sharper; however, if all we care about is showing that the percentage of numbers up to $x$ that are prime tends to zero then there's no need. This illustrates a general principle in mathematical proof writing: frequently one is lazy, if it works great; if the bound is too crude then we work harder!

%%%%%%%%%%%%%%%%%%%%%%%%%%%%%%%%%%%%%%%%%%%%%%%%%%%%%%%%%%%%%%%%%%%%%%%%%%%%%%%%%%%%%%%%%%%%%%%%%%%%%%%%%%%%%%%%%%%%%%%%%%%%%%%%%%%%%%%%%%%%%%%%%%%%%%%%%%%%%%%%%%%%%%%%%%%%%%%%%%%%%%%%%%%%%%%%%%%%%%%%%%%%%%%%%%%%%%%%%%%%%%%%%%%%%%%%%%%%%%%%%%%%%%%%%%%%%%%%%%%%%%%%%%%%%%%%%%%%%%%%%%%%%%%%%%%%%%%%%%%%%%%%%%%%%%%%%%%%%%%%%%%%%%%%%%%%%%%%%%%%%%%%%%%%%%%%%%%%%%%%%%%

\ \\
\end{document}